\documentclass[preprint,12pt]{elsarticle}

\usepackage{amssymb}
\usepackage{amsthm}
\usepackage{amsmath}
\usepackage{epic}
\usepackage{CJK}
\usepackage{setspace}
\newtheorem{thm}{Theorem}[section]
\newtheorem{cor}[thm]{Corollary}
\newtheorem{lem}[thm]{Lemma}

\newtheorem{example}[thm]{Example}

\journal{~}

\begin{document}
\begin{spacing}{1.15}
\begin{frontmatter}
\title{\textbf{The Minc-type bound and the eigenvalue inclusion sets of the general product of tensors}}

\author[label1]{Chunli Deng}
\author[label2]{Hongmei Yao}
\author[label1,label2]{Changjiang Bu}\ead{buchangjiang@hrbeu.edu.cn}


\address{
\address[label1]{College of Automation, Harbin Engineering University, Harbin 150001, PR China}
\address[label2]{College of Science, Harbin Engineering University, Harbin 150001, PR China}
}

\begin{abstract}
In this paper, we give the  Minc-type bound for spectral radius of nonnegative tensors. We also present the bounds for the spectral radius and the eigenvalue inclusion sets of the general product of tensors.

\end{abstract}

\begin{keyword}
General product of tensors, Spectral radius,  Minc-type bound, Eigenvalue inclusion set\\
\emph{AMS classification:}  15A69, 15A18
\end{keyword}
\end{frontmatter}

\section{Introduction}
Let $\mathbb{C}^{[m,n]}$ ($\mathbb{R}^{[m,n]}_{+}$) be the set of order $m$ dimension $n$ tensors (nonnegative tensors) over complex number field $\mathbb{C}$ (real number field $\mathbb{R}$).
For $\mathcal{A}=(a_{i_{1}i_{2}\cdots i_{m}}) \in {\mathbb{C}^{\left[ {m,n} \right]}}$, denote
\[ {r_i}(\mathcal{A}) =  {\sum\limits_{{i_2}, \ldots ,{i_m} =1}^{n} {{|a_{i{i_2} \cdots{i_m}}|}} } ,~{r(\mathcal{A})} = \mathop {\min }\limits_{i\in[n] } {r_i}(\mathcal{A}), {R(\mathcal{A})} = \mathop {\max }\limits_{i\in[n] }{r_i}(\mathcal{A}),\]
where $[n]=\{1,2,\ldots,n\}$.

For $\mathcal{A}=(a_{i_{1} \cdots i_{m}}) \in \mathbb{C}^{[m,n]}$, $x=(x_{1},\ldots,x_{n})^{\mathrm{T}} $$\in \mathbb{C}^{n}$, $\mathcal{A} x^{m-1}$ is a column vector of dimension $n$, whose the $i$th component is
\[(\mathcal{A} x^{m-1})_{i}=\sum\limits_{{i_2},\ldots ,{i_m} =1}^{n} {{a_{i{i_2} \cdots {i_m}}}{x_{{i_2}}} \cdots {x_{{i_m}}}},~i\in [n] .\]

In 2005, Qi \cite{Qi05} and Lim \cite{lim} proposed the concept of eigenvalue of tensors, independently.
For $\mathcal{A} \in {\mathbb{C}^{\left[ {m,n} \right]}}$, if there exists a number $\lambda\in \mathbb{C}$ and a nonzero  vector $x=(x_{1},\ldots,x_{n})^{\mathrm{T}} \in \mathbb{C}^{n}$ such that
\[\mathcal{A}{x^{m - 1}} = \lambda {x^{\left[ {m - 1} \right]}},\]
then $\lambda$ is called the eigenvalue of $\mathcal{A}$ and $x$ is called the eigenvector of $\mathcal{A}$ corresponding to $\lambda$, where $x^{[m-1]}=(x_{1}^{m-1}, \ldots ,x_{n}^{m-1})^{\mathrm{T}}$.
Let $\sigma(\mathcal{A})$ denote the set of eigenvalues of $\mathcal{A}$. The spectral radius of $\mathcal{A}$ is defined as $\rho(\mathcal{A})=\max\{|\lambda|:\lambda\in\sigma(\mathcal{A})\}$.

In 2013, Shao \cite{shao} introduced a general product of  tensors as follows: Let  $\mathcal{A}=(a_{i_{1}i_{2}\cdots i_{m}})\in \mathbb{C}^{[m,n]}$ $( m\geq 2)$ and $\mathcal{B}=(b_{i_{1}i_{2}\cdots i_{k}})\in \mathbb{C}^{[k,n]}$ $(k\geq1)$, then $\mathcal{A}\mathcal{B}=(c_{i\alpha_{1}\cdots\alpha_{m-1}})$ is
an  order $(m-1)(k-1)+1$  dimension $n$ tensor with entries:
\[c_{{i}\alpha_{1}\ldots\alpha_{m-1}}=\sum\limits_{{i_2}, \cdots ,{i_m} =1}^{n} {{a_{i{i_2} \cdots {i_m}}}{b_{{i_2}{\alpha _1}}} \cdots {b_{{i_m}{\alpha _{m - 1}}}}},i\in[n],\alpha_{1},\ldots, \alpha_{m-1}\in[n]^{k-1},\]
where ${[n]^{k-1}} =\underbrace {[n] \times  \cdots  \times [n]}_{k-1}$.
Clearly, when  $\mathcal{B}=x\in \mathbb{C}^{[1,n]}$, we have $\mathcal{A}x =\mathcal{A} x^{m-1}$. In this paper, we use $\mathcal{A}x$ instead of $\mathcal{A} x^{m-1}$.

A  tensor $\mathcal{A}=(a_{i_{1}i_{2}\cdots i_{m}}) \in \mathbb{R}_{+}^{[m,n]}$ is called  weakly irreducible \cite{yy2}, if for any  nonempty proper  subset $I\subset [n]$, there exist $i_{1},i_{2},\ldots,i_{m}$ satisfied
\[a_{i_{1}i_{2}\ldots i_{m}}>0,~~where~  i_{1}\in I,~~i_{j}\in [n]\setminus I,~j\in\{2,\ldots,m\}.\]

\vspace{1 ex}
In 1988, Minc \cite{minc} gave the bounds for  spectral radius of nonnegative matrices as follows: Let $A=(a_{ij})\in\mathbb{R}_{+}^{[2,n]}$, and $r_{i}(A)\neq 0$,~$i=1,\ldots ,n$. Then
\begin{align}\label{1}
\min_{i\in[n]}(\frac{1}{r_{i}(A)}\sum_{j=1}^{n}a_{ij}r_{j}(A))\leq \rho (A)\leq \max_{i\in[n]}(\frac{1}{r_{i}(A)}\sum_{j=1}^{n}a_{ij}r_{j}(A)) ,
\end{align}
and since $r_{i}(A^{2})=\sum\limits_{j=1}^{n}a_{ij}r_{j}(A)$,  the inequality (\ref{1}) could be written by
\[ \mathop {\min }\limits_{i\in[n] } \frac{{{r_i}({A^2})}}{{{r_i}(A)}} \leq \rho (A) \leq \mathop {\max }\limits_{i \in[n]} \frac{{{r_i}({A^2})}}{{{r_i}(A)}}.\]

In \cite{ch perron}, Chang et al. gave the Perron-Frobenius theorem for nonnegative irreducible tensor $\mathcal{A}$, and  the Collatz-Wielandt Theorem of $\mathcal{A}$ is given as follows:
\begin{align}\label{mini}
 \max\limits_{x\in \mathbb{R}^{n}_{++}} \mathop {\min }\limits_{i\in[n] } \frac{(\mathcal{A}x)_{i}}{x_{i}^{m-1}}=  \rho \left( \mathcal{A} \right) =\min\limits_{x\in \mathbb{R}_{++}^{n}}\mathop {\max }\limits_{i\in[n] } \frac{(\mathcal{A}x)_{i}}{x_{i}^{m-1}}.
\end{align}
 In \cite{s.f perron}, Friedland et al.  also showed the Collatz-Wielandt Theorem  for nonnegative weakly irreducible tensors. There are still some other results on  spectral radius of tensors, see \cite{ch perron,liwensome,sunlizhu}.

In \cite{Qi05}, Qi gave the Ger$\check{s}$gorin-type eigenvalue inclusion sets for real symmetric tensors. In \cite{bu}, Bu et al. gave the Brualdi-type eigenvalue inclusion sets of tensors via digraph.

Inspired by the above, in this paper, we give the  Minc-type bound for spectral radius of nonnegative tensors.  Further,  the generalized  Collatz-Wielandt Theorem for nonnegative weakly irreducible tensors is showed. We also present the bound for the spectral radius and the eigenvalue inclusion sets of the general product of tensors.

\section{Preliminaries}
In this section, we give some  lemmas.
\begin{lem}\label{ss}\cite{shao}
For $\mathcal{A}=(a_{i_{1}i_{2}\cdots i_{m}})\in \mathbb{C}^{[m,n]}$ and an invertible diagonal matrix $D=\mathrm{diag}(d_{11},d_{22},\ldots,d_{nn})$, $\mathcal{B}=D^{-(m-1)}\mathcal{A}D$ is an order $m$ dimension $n$ tensor with entries
\[b_{i_{1}i_{2}\cdots i_{m}}=d_{i_{1}i_{1}}^{-(m-1)}a_{i_{1}i_{2}\cdots i_{m}}d_{i_{2}i_{2}}\cdots d_{i_{m}i_{m}}.\]
In this case, $\mathcal{A}$ and $\mathcal{B}$ are called diagonal similar and $\sigma(\mathcal{A})=\sigma(\mathcal{B})$.
\end{lem}
\begin{lem}\label{pf}\cite{s.f perron}
Let $\mathcal{A}\in \mathbb{R}_{+}^{[m,n]}$ be a weakly irreducible tensor. Then $\rho(\mathcal{A})$ is an eigenvalue of $\mathcal{A}$, and there exists a unique positive eigenvector corresponding to $\rho(\mathcal{A})$ up to a multiplicative constant.
\end{lem}
\begin{lem}\label{bound}\cite{yy}
Let $\mathcal{A}\in \mathbb{R}_{+}^{[m,n]}$. Then
\[r(\mathcal{A})\leq \rho(\mathcal{A})\leq R(\mathcal{A}).\]
\end{lem}

\begin{lem}\label{2.4}
Let $\mathcal{A}=(a_{i_{1}i_{2}\cdots i_{m}})\in \mathbb{C}^{[m,n]}$, $ \mathcal{B}=(b_{i_{1}i_{2}\cdots i_{k}})\in \mathbb{C}^{[k,n]}$, $m\geq2,~k\geq1$.  Then\\$(1)$
\[r_{i}(\mathcal{A}\mathcal{B})\leq  {r_{i}(\mathcal{A})}(R(\mathcal{B}))^{m-1},i\in [n].\]
$(2)$If $\mathcal{A} $ and $\mathcal{B}$ are nonnegative tensors,
\[r_{i}(\mathcal{A}\mathcal{B})= \sum_{i_{2},\ldots,i_{m}=1}^{n}a_{ii_{2}\cdots i_{m}}r_{i_{2}}(\mathcal{B})\cdots r_{i_{m}}(\mathcal{B}),i\in [n].\]

\end{lem}
\begin{proof}

By general tensor product, we know  $\mathcal{A}\mathcal{B}=(c_{i\alpha_{1}\cdots\alpha_{m-1}})$ is an  order $(m-1)(n-1)+1$ dimension $n$ tensor, where
\[c_{{i}\alpha_{1}\cdots\alpha_{m-1}}=\sum\limits_{{i_2}, \ldots ,{i_m} =1}^{n} {{a_{i{i_2} \cdots {i_m}}}{b_{{i_2}{\alpha _1}}} \cdots {b_{{i_m}{\alpha _{m - 1}}}}}, ~\alpha_{j}\in [n]^{k-1},~j\in[m-1].\]
Then, for all $i\in [n]$, we have
\begin{eqnarray*}
{r_i}\left( {{\mathcal{A}\mathcal{B}}} \right) &=& \sum\limits_{{\alpha _1}, \ldots,{\alpha _{m - 1}} \in {{\left[ n \right]}^{k - 1}}} {\left| {{c_{i{\alpha _1} \cdots {\alpha _{m - 1}}}}} \right|} {\kern 1pt} {\kern 1pt} {\kern 1pt} {\kern 1pt} {\kern 1pt} {\kern 1pt} {\kern 1pt}  \\
&=& \sum\limits_{{\alpha _1},\ldots ,{\alpha _{m - 1}} \in {{\left[ n \right]}^{k - 1}}} {\left| {\sum\limits_{{i_2}, \ldots ,{i_m} =1}^{n} {{a_{i{i_2} \cdots {i_m}}}{b_{{i_2}{\alpha _1}}} \cdots } {b_{{i_m}{\alpha _{m - 1}}}}} \right|} \\
&\leq&
{\sum\limits_{{i_2}, \ldots ,{i_m} =1}^{n} {\left( {\sum\limits_{{\alpha _1},\ldots ,{\alpha _{m - 1}} \in {{\left[ n \right]}^{k - 1}}} {\left| {{a_{i{i_2} \cdots {i_m}}}} \right|\left| {{b_{{i_2}{\alpha _1}}}} \right| \cdots } \left| {{b_{{i_m}{\alpha _{m - 1}}}}} \right|} \right)} }\\
&=&\left( {\sum\limits_{{i_2}, \ldots,{i_m} =1}^{n} {\left| {{a_{i{i_2} \cdots {i_m}}}} \right|} } \left( {\sum\limits_{{\alpha _1},\ldots ,{\alpha _{m - 1}} \in {{\left[ n \right]}^{k - 1}}} {\left| {{b_{{i_2}{\alpha _1}}}} \right| \cdots } \left| {{b_{{i_m}{\alpha _{m - 1}}}}} \right|} \right)\right) \hfill ,\\
\end{eqnarray*}
Since
\begin{eqnarray*}
{\sum\limits_{{\alpha _1},\ldots,{\alpha _{m - 1}} \in {{\left[ n \right]}^{k - 1}}} {\left| {{b_{{i_2}{\alpha _1}}}} \right| \cdots } \left| {{b_{{i_m}{\alpha _{m - 1}}}}} \right|}
 &=&
(\sum\limits_{{\alpha _1} \in {{[n]}^{k - 1}}} {\left| {{b_{{i_2}{\alpha _1}}}} \right|}) (\sum\limits_{{\alpha _2} \in {{[n]}^{k - 1}}} {\left| {{b_{{i_3}{\alpha _2}}}} \right|})\\
&\cdots& ( \sum\limits_{{\alpha _{k - 1}} \in {{[n]}^{k - 1}}} {\left| {{b_{{i_m}{\alpha _{m - 1}}}}} \right|}) \\
&=&r_{i_{2}}(\mathcal{B})r_{i_{3}}(\mathcal{B}) \cdots r_{i_{m}}(\mathcal{B}) \\
&\leq&(R(\mathcal{B}))^{m-1},
\end{eqnarray*}
we have
\begin{eqnarray*}
{r_i}\left( {{\mathcal{A}\mathcal{B}}} \right) &\leq &\sum\limits_{{i_2}, \ldots ,{i_m} =1}^{n} {\left| {{a_{i{i_2} \cdots {i_m}}}} \right|}r_{i_{2}}(\mathcal{B})r_{i_{3}}(\mathcal{B}) \cdots r_{i_{m}}(\mathcal{B})\\
&\leq&(R(\mathcal{B}))^{m-1} \sum\limits_{{i_2}, \ldots ,{i_m} =1}^{n} {\left| {{a_{i{i_2} \cdots {i_m}}}} \right|}\\
&=&{r_{i}(\mathcal{A})}(R(\mathcal{B}))^{m-1}.
\end{eqnarray*}
So (1) of this lemma holds.

When $\mathcal{A} $ and $\mathcal{B}$ are nonnegative tensors, clearly, we obtain
\[r_{i}(\mathcal{A}\mathcal{B})= \sum_{i_{2},\ldots,i_{m}=1}^{n}a_{ii_{2}\cdots i_{m}}r_{i_{2}}(\mathcal{B})\cdots r_{i_{m}}(\mathcal{B}),\]
Thus we prove (2).
\end{proof}
\begin{cor}
Let $\mathcal{A}\in \mathbb{C}^{[m,n]}$. Then
\[R(\mathcal{A}^{k})\leq (R(\mathcal{A}))^{\mu_{k}}.\]
where  $\mu_{k}  = \left\{ {\begin{array}{*{20}{c}}
  {\begin{array}{*{20}{c}}
  {{{\frac{{{{(m - 1)}^k} - 1}}{{m - 2}}}}}&{,m > 2},
\end{array}} \\
  {\begin{array}{*{20}{c}}
  {{k}}&~~~~~~~~{,m = 2}.
\end{array}}
\end{array}} \right.$
\end{cor}

\section{The Minc-type bound for spectral radius of nonnegative tensors }

In this section, we give the bounds for spectral radius of nonnegative tensors and the generalized Collatz-Wielandt Theorem of  nonnegative weakly irreducible tensors.
\begin{thm}\label{mincb}
Let $\mathcal{A}\in{\mathbb{R}^{[m,n]}_{+}}$, $\mathcal{B}\in{\mathbb{R}^{[k,n]}_{+}}$, and $r_{i}(\mathcal{B})\neq0$, $i \in [n]$. Then
\[\mathop {\min }\limits_{i\in[n] } \frac{{{r_i}\left( {{\mathcal{A}\mathcal{B}}} \right)}}{({{r_i}\left( \mathcal{B} \right)})^{m-1}}\leq \rho \left( \mathcal{A} \right) \leq \mathop {\max }\limits_{i \in[n]} \frac{{{r_i}\left( {{\mathcal{A}\mathcal{B}}} \right)}}{({{r_i}\left( \mathcal{B} \right)})^{m-1}}.\]
\end{thm}
\begin{proof}
Let $D = \mathrm{diag}\left( {{r_1}\left( \mathcal{B} \right),{r_2}\left( \mathcal{B} \right), \ldots ,{r_n}\left( \mathcal{B} \right)} \right)$ be an invertible diagonal matrix. By Lemma \ref{ss}, ${{D^{ - (m-1)}}\mathcal{A}D}$ and  $\mathcal{A}$ are diagonal similar, then $ \rho \left( \mathcal{A} \right) = \rho \left( {{D^{ - (m-1)}}\mathcal{A}D} \right)$. From Lemma \ref{bound}, we get
\[  {r}\left( {{D^{ - (m-1)}}\mathcal{A}D} \right)\leq \rho \left( \mathcal{A} \right) = \rho \left( {{D^{ - (m-1)}}\mathcal{A}D} \right) \leq  {R}\left( {{D^{ - (m-1)}}\mathcal{A}D} \right) .\]
Using the general product of tensors, we have
\[{\left( {{D^{ - (m-1)}}\mathcal{A}D} \right)_{{i_1}{i_2} \cdots {i_m}}} = {a_{{i_1}{i_2} \cdots {i_m}}}{\left( {{r_{{i_1}}}\left( \mathcal{B} \right)} \right)^{ - (m-1)}}{r_{{i_2}}}\left( \mathcal{B} \right) \cdots {r_{{i_m}}}\left( \mathcal{B} \right),\]
then
\begin{eqnarray*}
  {r_i}\left( {{D^{ - (m-1)}}\mathcal{A}D} \right) &=& \sum\limits_{{i_2}, \ldots ,{i_m} =1}^{n} {{a_{i{i_2} \cdots {i_m}}}{{\left( {{r_i}\left( \mathcal{B} \right)} \right)}^{ - (m-1)}}{r_{{i_2}}}\left(\mathcal{ B} \right) \cdots {r_{{i_m}}}\left( \mathcal{B }\right)}  \hfill \\
   &=&\frac{{\sum\limits_{{i_2}, \ldots ,{i_m} =1}^{n} {{a_{i{i_2} \cdots {i_m}}}{r_{{i_2}}}\left( \mathcal{B} \right) \cdots {r_{{i_m}}}\left( \mathcal{B} \right)} }}{({{r_i}\left( \mathcal{B} \right)})^{m-1}} \hfill .\\
\end{eqnarray*}
Thus
\begin{align}\label{3}
\mathop {\min }\limits_{i \in[n]} \frac{{\sum\limits_{{i_2}, \ldots ,{i_m} =1}^{n} {{a_{i{i_2} \cdots {i_m}}}\prod\limits_{j=2}^{m}{r_{{i_j}}}\left( \mathcal{B} \right)  } }}{({{r_i}\left( \mathcal{B} \right)})^{m-1}}\leq \rho \left( \mathcal{A} \right) \leq \mathop {\max }\limits_{i \in[n]}  \frac{{\sum\limits_{{i_2}, \ldots ,{i_m} =1}^{n} {{a_{i{i_2} \cdots {i_m}}}\prod\limits_{j=2}^{m}{r_{{i_j}}}\left( \mathcal{B} \right)  } }}{({{r_i}\left( \mathcal{B} \right)})^{m-1}}.
\end{align}
By Lemma \ref{2.4}, it yields
\[r_{i}(\mathcal{A}\mathcal{B})= \sum\limits_{i_{2},\ldots,i_{m}=1}^{n}a_{ii_{2}\cdots i_{m}}r_{i_{2}}(\mathcal{B})\cdots r_{i_{m}}(\mathcal{B}),\]
combining inequality (\ref{3}), we have
\[
\mathop {\min }\limits_{i\in[n] } \frac{{{r_i}\left( {{\mathcal{A}\mathcal{B}}} \right)}}{({{r_i}\left( \mathcal{B} \right)})^{m-1}}\leq \rho \left( \mathcal{A} \right) \leq \mathop {\max }\limits_{i \in[n]} \frac{{{r_i}\left( {{\mathcal{A}\mathcal{B}}} \right)}}{({{r_i}\left( \mathcal{B} \right)})^{m-1}}.
\]
\end{proof}
\noindent\textbf{Remark }:In Theorem \ref{mincb}, we consider  two cases as follows.\\
\textbf{Case 1}. When $k=1$, let $\mathcal{B}=x=(x_{1},\ldots,x_{n})\in{\mathbb{R}^{n}_{++}}$ ( ${\mathbb{R}_{++}}$ is the set of positive  numbers),  $r_{i}(\mathcal{A}x)$ and $r_{i}(x)$  be written by $(\mathcal{A}x)_{i}$ and $x_{i}$, respectively. Thus we obtain
\[\mathop {\min }\limits_{i \in [n]} \frac{{{{(\mathcal{A}x)}_i}}}{{x_i^{m - 1}}} \leqslant \rho (\mathcal{A}) \leqslant \mathop {\max }\limits_{i \in [n]} \frac{{{{(\mathcal{A}x)}_i}}}{{x_i^{m - 1}}},\]
which is given in \cite{yy}.\\
\textbf{Case 2}. When $k=2$, let  $\mathcal{B}$ be an identity matrix, then
\[ r(\mathcal{A}) \leq \rho(\mathcal{A})\leq R(\mathcal{A}),\]
which is also given in \cite{yy}.

\vspace{2 ex}
If we set $\mathcal{B}=\mathcal{A}$ in Theorem \ref{mincb}, we obtain the following theorem, which extends the Minc-type bound for spectral radius  to tensors.
\begin{thm}\label{3.2}
Let $\mathcal{A}=(a_{i_{1}i_{2}\cdots i_{m}})\in{\mathbb{R}^{[m,n]}_{+}}$. If $r_{i}(\mathcal{A})\neq0$, for all $i\in [n] $, then
\[\mathop {\min }\limits_{i\in[n] } \frac{{{r_i}\left( {{\mathcal{A}^2}} \right)}}{({{r_i}\left( \mathcal{A} \right)})^{m-1}}\leq \rho \left( \mathcal{A} \right) \leq \mathop {\max }\limits_{i \in[n]} \frac{{{r_i}\left( {{\mathcal{A}^2}} \right)}}{({{r_i}\left( \mathcal{A} \right)})^{m-1}}.\]
\end{thm}

\begin{example}
Let $\mathcal{A} \in {\mathbb{R}_{+}^{\left[ {3,2} \right]}}$, where $a_{111}=3, ~a_{112}=1,~ a_{121}=2,~ a_{122}=1, ~a_{211}=0, ~a_{212}=4, a_{221}=2,~ a_{222}=3 $. Then
 \[r_{1}(\mathcal{A})=7,~ r_{2}(\mathcal{A})=9,~ r_{1}(\mathcal{A}^{2})=417,~ r_{2}(\mathcal{A}^{2})=621,\]
\[\frac{r_{1}(\mathcal{A}^{2})}{(r_{1}(\mathcal{A}))^{2}}=\frac{417}{49},~\frac{r_{2}(\mathcal{A}^{2})}{(r_{2}(\mathcal{A}))^{2}}=\frac{621}{81}.\]
From Lemma $\ref{bound}$, we have
\[7\leq \rho(\mathcal{A})\leq 9.\]
From Theorem $\ref{3.2}$, we have
\[7.6667\approx\frac{621}{81}\leq \rho(\mathcal{A})\leq \frac{417}{49}\approx8.5102.\]
This example shows that the bound in Theorem $\ref{3.2}$ is better than the bound in Lemma $\ref{bound}$ for some tensors.

\end{example}

\vspace{1 ex}
We take $\mathcal{B}=\mathcal{A}^{k}$ in Theorem \ref{mincb}, we obtain the following theorem.

\begin{thm}\label{3.4}
Let $\mathcal{A}=(a_{i_{1}i_{2}\cdots i_{m}})\in \mathbb{R}^{[m,n]}_{+}$. If $r_{i}(\mathcal{A})\neq0$,  for all $i\in[n] $, then
\[\mathop {\min }\limits_{i \in[n]} \frac{{{r_i}\left( {{\mathcal{A}^{k+1}}} \right)}}{({{r_i}\left( \mathcal{A}^{k} \right)})^{m-1}}\leq \rho \left( \mathcal{A} \right) \leq \mathop {\max }\limits_{i\in[n] } \frac{{{r_i}\left( {{\mathcal{A}^{k+1}}} \right)}}{({{r_i}\left( \mathcal{A}^{k} \right)})^{m-1}}.\]
\end{thm}

\noindent\textbf{Remark }:Let $\mathcal{A}=(a_{i_{1}i_{2}\cdots i_{m}})\in \mathbb{R}^{[m,n]}_{+}$ and $r_{i}(\mathcal{A})\neq0$, for all $i\in[n] $. It follows from Lemma \ref{2.4} that ${r_i}({\mathcal{A}^{2}}) = \sum\limits_{{i_2}, \ldots ,{i_m} = 1}^n {{a_{i{i_2} \cdots {i_m}}}{r_{{i_2}}}({\mathcal{A}}) \cdots {r_{{i_m}}}({\mathcal{A}})}$. Since $r_{i}(\mathcal{A})\neq 0$,  we have $r_{i}(\mathcal{A}^{2})\neq 0$. Similarly, we obtain $r_{i}(\mathcal{A}^{3})\neq 0,\cdots,r_{i}(\mathcal{A}^{k})\neq 0$, for all $i \in [n]$. Thus, in Theorem \ref{3.4},
 $({{r_i}\left( \mathcal{A}^{k} \right)})^{m-1}\neq 0$, for all $i \in [n]$.

\vspace{1 ex}
Next we give the generalized Collatz-Wielandt Theorem of nonnegative weakly irreducible tensors.
\begin{thm}\label{3.5}
Let $\mathcal{A}\in{\mathbb{R}^{[m,n]}_{+}}$ be a weakly irreducible tensor. Then
\[\max_{\mathcal{B}\in{\mathbb{R}^{[k,n]}_{+}},r_{i}(\mathcal{B})\neq0} \mathop {\min }\limits_{i\in[n] } \frac{{{r_i}\left( {{\mathcal{A}\mathcal{B}}} \right)}}{({{r_i}\left( \mathcal{B} \right)})^{m-1}}=  \rho \left( \mathcal{A} \right) = \min_{\mathcal{B}\in{\mathbb{R}^{[k,n]}_{+}},r_{i}(\mathcal{B})\neq0} \mathop {\max }\limits_{i \in[n]} \frac{{{r_i}\left( {{\mathcal{A}\mathcal{B}}} \right)}}{({{r_i}\left( \mathcal{B} \right)})^{m-1}}.\]
where $k$ is any fixed positive integer.
\end{thm}
\begin{proof}
Since $\mathcal{A}$ is a nonnegative irreducible tensor, by Lemma \ref{pf},
$\rho \left( \mathcal{A} \right)$ is an eigenvalue of $\mathcal{A}$ and $x=(x_{1},\ldots ,x_{1})\in\mathbb{R}_{++}^{n}$ is the eigenvector corresponding to $\rho(\mathcal{A})$.
Then
\[\rho \left( \mathcal{A} \right)=\frac{{\sum\limits_{{i_2}, \ldots ,{i_m} =1}^{n} {{a_{i{i_2} \cdots {i_m}}}{x_{{i_2}}} \cdots {x_{{i_m}}}} }}{{{x_i}}^{m-1}},~for~all~i\in[n].\]
For any positive integer $k$, there exists a tensor $\mathcal{B}\in{\mathbb{R}^{[k,n]}_{+}}$, such that ${{r_i}\left( {{\mathcal{B}}} \right)}=x_{i}$, thus
 \[\rho \left( \mathcal{A} \right)=\frac{{\sum\limits_{{i_2},\ldots,{i_m} =1}^{n} {{a_{i{i_2} \cdots {i_m}}}{r_{{i_2}}}\left( \mathcal{B} \right) \cdots {r_{{i_m}}}\left( \mathcal{B} \right)} }}{({{r_i}\left( \mathcal{B} \right)})^{m-1}}=\frac{{{r_i}\left( {{\mathcal{A}\mathcal{B}}} \right)}}{({{r_i}\left( \mathcal{B} \right)})^{m-1}},\]
 by Theorem \ref{mincb}, we know
 \[\mathop {\min }\limits_{i\in[n] } \frac{{{r_i}\left( {{\mathcal{A}\mathcal{B}}} \right)}}{({{r_i}\left( \mathcal{B} \right)})^{m-1}}\leq \rho \left( \mathcal{A} \right) \leq \mathop {\max }\limits_{i \in[n]} \frac{{{r_i}\left( {{\mathcal{A}\mathcal{B}}} \right)}}{({{r_i}\left( \mathcal{B} \right)})^{m-1}}.\]
Then
\[\max_{\mathcal{B}\in{\mathbb{R}^{[k,n]}_{+}},r_{i}(\mathcal{B})\neq0} \mathop {\min }\limits_{i\in[n] } \frac{{{r_i}\left( {{\mathcal{A}\mathcal{B}}} \right)}}{({{r_i}\left( \mathcal{B} \right)})^{m-1}}=  \rho \left( \mathcal{A} \right) = \min_{\mathcal{B}\in{\mathbb{R}^{[k,n]}_{+}},r_{i}(\mathcal{B})\neq0} \mathop {\max }\limits_{i \in[n]} \frac{{{r_i}\left( {{\mathcal{A}\mathcal{B}}} \right)}}{({{r_i}\left( \mathcal{B} \right)})^{m-1}}.\]
\end{proof}
\noindent\textbf{Remark }:In Theorem \ref{3.5}, when $\mathcal{B}=x\in\mathbb{C}^{[1,n]}$, we obtain Collatz-Wielandt Theorem for nonnegative weakly irreducible tensors, which is showed in equality (\ref{mini}).

\section{ The eigenvalue inclusion sets of the general product of tensors   }
In this section, the bounds for spectral radius and the eigenvalue inclusion sets of the general product of two tensors are discussed.

\begin{thm}\label{4.1}
Let $\mathcal{A}\in \mathbb{R}^{[m,n]}_{+},~ \mathcal{B}\in \mathbb{R}^{[k,n]}_{+}$. Then
\[ r(\mathcal{A})( r(\mathcal{B}))^{m-1}\leq \rho(\mathcal{A}\mathcal{B})\leq R(\mathcal{A})(R(\mathcal{B}))^{m-1}.\]
\end{thm}
\begin{proof}
Let $\mathcal{A}=(a_{i_{1}i_{2}\ldots i_{m}})\in \mathbb{R}^{[m,n]}_{+}$, $\mathcal{B}\in \mathbb{R}^{[k,n]}_{+}$. By Lemma \ref{2.4}, we have \[r_{i}(\mathcal{A}\mathcal{B})= \sum\limits_{i_{2},\ldots,i_{m}=1}^{n}a_{ii_{2}\ldots i_{m}}r_{i_{2}}(\mathcal{B})\cdots r_{i_{m}}(\mathcal{B}) .\]
Clearly, $r_{i}(\mathcal{A}\mathcal{B}) \geq r_{i}(\mathcal{A})( r(\mathcal{B}))^{m-1}$, for all $i\in[n]$.
Let $ r_{j}(\mathcal{A}\mathcal{B})=r(\mathcal{A}\mathcal{B}),~j\in[n]$. Then we get
\[r(\mathcal{A}\mathcal{B})= r_{j}(\mathcal{A}\mathcal{B}) \geq r_{j}(\mathcal{A})( r(\mathcal{B}))^{m-1}\geq r(\mathcal{A})( r(\mathcal{B}))^{m-1}.\]
It follows from  Lemma \ref{bound} that $\rho(\mathcal{A}\mathcal{B})\geq  r(\mathcal{A}\mathcal{B})$. So we obtain \[\rho(\mathcal{A}\mathcal{B})\geq  r(\mathcal{A})(  r(\mathcal{B}))^{m-1}.\]
Similarly, we have $ \rho(\mathcal{A}\mathcal{B})\leq R(\mathcal{A}\mathcal{B})\leq R(\mathcal{A})(R(\mathcal{B}))^{m-1}$. Thus \[ r(\mathcal{A})( r(\mathcal{B}))^{m-1}\leq \rho(\mathcal{A}\mathcal{B})\leq R(\mathcal{A})(R(\mathcal{B}))^{m-1}.\]

\end{proof}

\noindent\textbf{Remark }:In Theorem \ref{4.1}, when $\mathcal{B}$ is an identity matrix,  we  also obtain $ r(\mathcal{A}) \leq \rho(\mathcal{A})\leq R(\mathcal{A})$ (see \cite{yy}).

\vspace{1 ex}
From Theorem \ref{4.1} and Corollary 2.5, we have the following results.
\begin{cor}
Let $\mathcal{A}\in \mathbb{C         }^{[m,n]}$. Then\\
$(1)$ $  \rho(\mathcal{A}^{k})\leq (R(\mathcal{A}))^{\mu_{k}};$\\
$(2)$ If $\mathcal{A}$ is a nonnegative tensor, we have
\[{(r(A))^{\mu_{k}}} \leq \rho(\mathcal{A}^{k})\leq (R(\mathcal{A}))^{\mu_{k}},\]
where  $\mu_{k}  = \left\{ {\begin{array}{*{20}{c}}
  {\begin{array}{*{20}{c}}
  {{{\frac{{{{(m - 1)}^k} - 1}}{{m - 2}}}}}&{,m > 2},
\end{array}} \\
  {\begin{array}{*{20}{c}}
  {{k}}&~~~~~~~~{,m = 2}.
\end{array}}
\end{array}} \right.$
\end{cor}
In\cite{Qi05}, Qi gave the Ger$\check{s}$gorin-type eigenvalue inclusion sets of tensors, i.e.,
\begin{align}\label{4}
 \sigma(\mathcal{A})\subseteq \bigcup_{i\in [n]}\{z\in \mathbb{C}: |z-a_{i\cdots i}|\leq r _{i}(\mathcal{A})-|a_{i\cdots i}|\}.
\end{align}
  For  $\mathcal{A}\in \mathbb{C}^{[m,n]}$ and  $\mathcal{B}\in \mathbb{C}^{[k,n]}$. By Lemma \ref{2.4}, it yields
\begin{align}\label{5}
  r _{i}(\mathcal{AB})\leq r_{i}(\mathcal{A})(R(\mathcal{B}))^{m-1}.
\end{align}
Denote by $c_{i\cdots i}$ the diagonal elements of tensor $\mathcal{AB}$, then
\begin{align}\label{6}
c_{i\cdots i}=\sum\limits_{i_{2},\ldots ,i_{m}=1 }^{n}a_{ii_{2}\cdots i_{m}}b_{i_{2}i\cdots i}\cdots b_{i_{m}i\cdots i},~i\in[n].
\end{align}
 Combining (\ref{4}), (\ref{5}), (\ref{6}),  we  obtain the eigenvalue inclusion sets for $\mathcal{AB}$ as follows.

\begin{thm}\label{4.3}
Let $\mathcal{A}=(a_{i_{1}i_{2}\cdots i_{m}})\in \mathbb{C}^{[m,n]}$ and $\mathcal{B}=(b_{i_{1}i_{2}\cdots i_{k}})\in \mathbb{C}^{[k,n]}$. Then
\[\sigma(\mathcal{AB})\subseteq  \mathbf{G}= \bigcup_{i\in [n]}\{z\in \mathbb{C}: |z-c_{i\cdots i}|\leq r _{i}(\mathcal{A})(R(\mathcal{B}))^{m-1}-|c_{i\cdots i}|\},\]
where $c_{i\cdots i}=\sum\limits_{i_{2},\ldots, i_{m}=1 }^{n}a_{ii_{2}\cdots i_{m}}b_{i_{2}i\cdots i}\cdots b_{i_{m}i\cdots i}.$
\end{thm}

For a tensor $\mathcal{A}=(a_{i_1\cdots i_m })\in\mathbb{C}^{[m,n]}$, we associate with $\mathcal{A}$ a directed graph $\Gamma_{\mathcal{A}}$ as follows, $\Gamma_{\mathcal{A}}=(V(\mathcal{A}),E(\mathcal{A }))$, where $V(\mathcal{A})=\left\{1,\ldots,n\right\}$ is  vertex set and $E(\mathcal{A }) = \{(i,j):a_{ii_2\cdots i_m}\neq0,j\in \left\{ {{i_2}, \ldots ,{i_m}} \right\}\neq \left\{ {{i}, \ldots ,{i}} \right\}\}$ is arc set(see \cite{kpearson}, \cite{s.f perron}). If for each vertex $i \in V(\mathcal{A})$, there exists a circuit $\gamma$, such that $i$ belong to $\gamma$, then $\Gamma_{\mathcal{A}}$ is called weakly connected. We denote the set of circuits of $\Gamma_{\mathcal{A}}$ by $C(\mathcal{A})$.

\vspace{1 ex}
In \cite{bu}, Bu et al. gave the Brualdi-type eigenvalue inclusion sets via digraph as follows: Let $\mathcal{A}=(a_{i_{1}i_{2}\ldots i_{m}})\in \mathbb{C}^{[m,n]}$. If  $\Gamma_{\mathcal{A}}$ is weakly connected, then
\begin{align} \label{7}
\sigma(\mathcal{A})\subseteq \bigcup\limits_{\gamma\in C(\mathcal{A})}\{z\in \mathbb{C}: \prod\limits_{i\in \gamma}|z-a_{i\cdots i}|\leq \prod\limits_{i\in \gamma}(r_{i}(\mathcal{A})-|c_{i\cdots i}|)\}.
\end{align}
Next we give another eigenvalue inclusion sets for tensor $\mathcal{AB}$.

\begin{thm}\label{4.4}
Let $\mathcal{A}=(a_{i_{1}i_{2}\cdots i_{m}})\in \mathbb{C}^{[m,n]}$, $ \mathcal{B}=(b_{i_{1}i_{2}\cdots i_{k}})\in \mathbb{C}^{[k,n]}$ and $\Gamma_{\mathcal{A}\mathcal{B}}$ be weakly connected.  Then
\[\sigma(\mathcal{AB})\subseteq \mathbf{B}=\bigcup_{\gamma\in C(\mathcal{A}\mathcal{B})}\{z\in \mathbb{C}: \prod_{i\in \gamma}|z-c_{i\cdots i}|\leq \prod_{i\in \gamma}(r_{i}(\mathcal{A})(R(\mathcal{B}))^{m-1}-|c_{i\cdots i}|)\}\subseteq \mathbf{G},\]
where $c_{i\cdots i}=\sum\limits_{i_{2},\ldots ,i_{m}=1 }^{n}a_{ii_{2}\ldots i_{m}}b_{i_{2}i\cdots i}\cdots b_{i_{m}i\cdots i},$ $\mathbf{G}$ is given in Theorem $\ref{4.3}$.
\end{thm}
\begin{proof}
Since $\mathcal{A}=(a_{i_{1}i_{2}\cdots i_{m}})\in \mathbb{C}^{[m,n]}$, $ \mathcal{B}=(b_{i_{1}i_{2}\cdots i_{k}})\in \mathbb{C}^{[k,n]}$ and $\Gamma_{\mathcal{A}\mathcal{B}}$ is weakly connected. Then combining (\ref{5}), (\ref{6}), (\ref{7}),  we  obtain
\[\sigma(\mathcal{AB})\subseteq \mathbf{B}=\bigcup_{\gamma\in C(\mathcal{A}\mathcal{B})}\{z\in \mathbb{C}: \prod_{i\in \gamma}|z-c_{i\cdots i}|\leq \prod_{i\in \gamma}(r_{i}(\mathcal{A})(R(\mathcal{B}))^{m-1}-|c_{i\cdots i}|)\}.\]
For any $z\in \mathbf{B} $, if $z \notin \mathbf{G} $, then \[|z-c_{i\cdots i}|\leq r _{i}(\mathcal{A})(R(\mathcal{B}))^{m-1}-|c_{i\cdots i}|, ~for ~all~ i \in [n].\]
Thus \[ \prod\limits_{i\in \gamma}|z-c_{i\cdots i}|> \prod\limits_{i\in \gamma}(r_{i}(\mathcal{A})(R(\mathcal{B}))^{m-1}-|c_{i\cdots i}|), ~for ~all~ \gamma \in C(\mathcal{A}), \]
this is a contradiction to $z\in \mathbf{B} $.
Therefore $z \in \mathbf{G} $, i.e., $\mathbf{B} \subseteq \mathbf{G}$.
\end{proof}

\end{spacing}
\end{document}